\documentclass[11pt,a4paper]{amsart}
\usepackage{exscale,amssymb,amsmath,color,amsmath,amsthm,amsfonts,bbm,stmaryrd}
\usepackage{graphicx,fullpage}
\usepackage{hyperref}
\usepackage{tikz}
\usetikzlibrary{shapes,shapes.misc,calc,intersections,patterns,decorations.markings,plotmarks}
   \def\pathscale{1.2}

\tikzstyle{vertex}=[circle,fill=black!75,minimum size=5pt,inner sep=0pt]
\tikzstyle{car}=[minimum size=8pt,inner sep=0pt]

\newcommand{\E}[1]{\ensuremath{\mathbb{E} \left[#1 \right]}}

\newcommand{\Prob}[1]{\ensuremath{\mathbb{P} \left(#1 \right)}}

\newcommand{\bP}{\ensuremath{\mathbb{P}}}

\newcommand{\var}[1]{\ensuremath{\mathrm{var} \left(#1 \right)}}

\newcommand{\R}{\ensuremath{\mathbb{R}}}
\newcommand{\Z}{\ensuremath{\mathbb{Z}}}
\newcommand{\NN}{\ensuremath{\mathbb{N}}}

\newcommand{\Po}{\ensuremath{\mathrm{Po}}}
\newcommand{\PGW}{\ensuremath{\mathrm{PGW}}}
\newcommand{\PGWinf}{\ensuremath{\mathrm{PGW}_{\infty}}}

\newcommand{\fl}[1]{\ensuremath{\lfloor #1 \rfloor}}

\renewcommand{\subset}{\subseteq}
\newcommand{\convdist}{\ensuremath{\stackrel{d}{\longrightarrow}}}
\newcommand{\convprob}{\ensuremath{\stackrel{p}{\rightarrow}}}
\newcommand{\equidist}{\ensuremath{\stackrel{d}{=}}}

\newenvironment{proofOfTheorem}[1]{\begin{proof}[\textit{Proof of Theorem \ref{#1}}]}{\end{proof}}

\newenvironment{itemize*}%
  {\vspace{-0.3cm}%
  \begin{itemize}%
    \setlength{\itemsep}{0pt}%
    \setlength{\parskip}{0pt}}%
  {\end{itemize}}

\newenvironment{enumerate*}%
  {\vspace{-0.3cm}%
  \begin{enumerate}%
    \setlength{\itemsep}{0pt}%
    \setlength{\parskip}{0pt}}%
  {\end{enumerate}}

\newtheorem{thm}{Theorem}[section]
\newtheorem{lemma}[thm]{Lemma}
\newtheorem{fact}[thm]{Fact}

\newtheorem{prop}[thm]{Proposition}
\newtheorem{cor}[thm]{Corollary}
\newtheorem{conj}[thm]{Conjecture}

\title{Parking on a random tree}
\author{Christina Goldschmidt}
\address{Department of Statistics and Lady Margaret Hall, University of Oxford}
\email{goldschm@stats.ox.ac.uk}

\author{Micha{\l} Przykucki}
\address{Mathematical Institute and St Anne's College, University of Oxford}
\email{przykucki@maths.ox.ac.uk}

\date{\today}

\begin{document}

\begin{abstract}
Consider a uniform random rooted tree on vertices labelled by $[n] = \{1,2,\ldots,n\}$, with edges directed towards the root.  We imagine that each node of the tree has space for a single car to park.  A number $m \le n$ of cars arrive one by one, each at a node chosen independently and uniformly at random.  If a car arrives at a space which is already occupied, it follows the unique path oriented towards the root until it encounters an empty space, in which case it parks there; if there is no empty space, it leaves the tree.  Consider $m = \fl{\alpha n}$ and let $A_{n,\alpha}$ denote the event that all $\fl{\alpha n}$ cars find spaces in the tree.  Lackner and Panholzer ~\cite{LacknerPanholzer} proved (via analytic combinatorics methods) that there is a phase transition in this model.   Then if $\alpha \le 1/2$, we have $\Prob{A_{n,\alpha}} \to \frac{\sqrt{1-2\alpha}}{1-\alpha}$, whereas if $\alpha > 1/2$ we have $\Prob{A_{n,\alpha}} \to 0$.  We give a probabilistic explanation for this phenomenon, and an alternative proof via the objective method.  Along the way, we are led to consider the following variant of the problem: take the tree to be the family tree of a Galton-Watson branching process with Poisson(1) offspring distribution, and let an independent Poisson($\alpha$) number of cars arrive at each vertex. Let $X$ be the number of cars which visit the root of the tree.  Then for $\alpha \le 1/2$, we have $\E{X} \leq 1$, whereas for $\alpha > 1/2$, we have $\E{X} = \infty$.  This discontinuous phase transition turns out to be a generic phenomenon in settings with an arbitrary offspring distribution of mean at least 1 for the tree and arbitrary arrival distribution.
\end{abstract}

\maketitle

\section{Introduction} \label{sec:intro}

Let $\Pi_n$ be the directed path on $[n] = \{1,2,\ldots,n\}$ with edges directed from $i+1$ to $i$ for $i=1,2,\ldots,n-1$. Let $m \leq n$ and assume that $m$ cars arrive at the path in some order, with the $i$th driver wishing to park in the spot $s_i \in [n]$. If a driver finds their preferred parking spot empty, they stop there. If not, they drive along the path towards $1$, taking the first available place. If no such place is found, they leave the path without parking. If all drivers find a place to park then we call $(s_1, s_2, \ldots, s_m)$ a \emph{parking function for $\Pi_n$}.

Konheim and Weiss \cite{KonheimWeiss-parkingPath} introduced parking functions in the context of collisions of hashing functions. Imagine that we have a \emph{hash table} consisting of a linear array of $n$ cells, where we want to store $m$ items. We use a \emph{hashing function} $h:[m] \to [n]$ to determine where each item is stored. Item $i$ is stored in cell $h(i)$, unless some item $j < i$ has already occupied it , in which case we have a \emph{collision}. We can resolve a collision by allocating item $i$ to the smallest cell $k > h(i)$ such that $k$ is empty at time $i$, if such a cell can be found. If not, our scheme fails, and we cannot allocate our items to the hashing table. This collision resolving scheme is clearly modelled by the parking functions described in the first paragraph.

Konheim and Weiss showed that for $1 \leq m \leq n$ cars there exist exactly $(n+1-m)(n+1)^{m-1}$ parking functions for $\Pi_n$. Hence, taking $\alpha \in (0,1)$ and $m = \lfloor \alpha n \rfloor$,  if the $i$th driver independently picks a uniformly random preferred parking spot $S_i$ then the probability that $(S_1,S_2, \ldots, S_m)$ is a parking function for $\Pi_n$ is
\[
  \frac{(n+1-m)(n+1)^{m-1}}{n^m} \to (1-\alpha) e^\alpha,
\]
as $n \to \infty$. In particular, this limiting probability is strictly positive for every $\alpha \in (0,1)$.

Some generalisations of parking functions and their connections to other combinatorial objects have been studied by, for example, Stanley \cite{Stanley-hyperplanesAndTrees, Stanley-parkingPartitions, Stanley-enumerativeCombinatorics, Stanley-hyperplanesAndParking}. In a recent paper, Lackner and Panholzer \cite{LacknerPanholzer} studied parking functions on other directed graphs, in particular on uniform random rooted labelled trees (uniform random rooted Cayley trees). Let $T_n$ denote such a tree on $n$ vertices. Each of the $m$ cars independently picks a uniform vertex and tries to park at it. If it is already occupied, the car moves towards the root and parks at the first empty vertex it encounters. If it finds no empty vertex, it leaves the tree. Lackner and Panholzer (see Theorem 4.10 and Corollary 4.11 in \cite{LacknerPanholzer}) prove that in this setting there is a phase transition.

\begin{thm} \label{thm:L&P}
Let $T_n$ denote a uniform random rooted labelled tree on $n$ vertices. Let $A_{n,\alpha}$ be the event that all $\lfloor \alpha n \rfloor$ cars, with uniform and independent random preferred parking spots, can park on $T_n$. Then 
\[
 \lim_{n \to \infty} \bP(A_{n,\alpha}) = 
 \begin{cases} 
  \frac{\sqrt{1-2\alpha}}{1-\alpha} & \text{ if $0 \le \alpha \le 1/2$,} \\
  0 & \text{ if $\alpha > 1/2$}.
 \end{cases}
\]
\end{thm}
In fact, the result proved in \cite{LacknerPanholzer} is much sharper: it not only demonstrates that there is a phase transition, but it also gives an asymptotic formula for $\bP(A_{n,\alpha})$ which specifies its behaviour in $n$, including at the critical point  $\alpha = 1/2$. However, the analytic methods used in \cite{LacknerPanholzer} offer no explanation for \emph{why} the phase transition occurs.  The purpose of the present paper is to find a probabilistic explanation for this phenomenon.  We employ the objective method, pioneered by Aldous and Steele~\cite{AldousSteele-LocalConvergence}, to reprove Theorem~\ref{thm:L&P}.  Much of our analysis is performed in the context of a limiting version of the above model (its so-called \emph{local weak limit}).  Instead of $T_n$, we consider a critical Galton-Watson tree with Poisson mean $1$ offspring distribution, conditioned on non-extinction.  We replace the multinomial counts of cars wishing to park at each vertex by independent Poisson mean $\alpha$ numbers of cars at each vertex. Once we have analysed this limiting model, it is relatively straightforward to then show that the probability all cars can park really gives the limit of $\Prob{A_{n,\alpha}}$ as $n \to \infty$.

\subsection{The limiting model}
Throughout this paper we write $\Po(\alpha)$ for the Poisson distribution with mean $\alpha$.  Write $\PGW(\alpha)$ for the law of the family tree of a Galton--Watson branching process with $\Po(\alpha)$ offspring distribution (this is canonically thought of as an ordered tree rooted at the progenitor of the branching process, although we shall frequently ignore the ordering).  We begin by formally introducing our limiting model.

Let $T$ be an infinite random tree defined as follows. Start with an infinite directed path $\Pi_\infty$ on $\NN = \{1,2,\ldots\}$, with edges directed from $n+1$ to $n$ for all $n \geq 1$. Then, for every $n$, add an independent $\PGW(1)$ tree rooted at $n$, with edges directed towards $n$ (see Figure \ref{fig:treeLimit}).  Finally, root the resulting (infinite) tree at 1.  This random tree has the same law as a $\PGW(1)$ tree conditioned on non-extinction, and we will write $\PGWinf(1)$ for its law. (Since extinction occurs with probability 1, the conditioning must be obtained by a limiting procedure such as conditioning the tree to survive to generation $k$ and then letting $k \to \infty$; see Kesten~\cite{Kesten}.  We will discuss a more general case of this result in Theorem~\ref{thm:Kesten} below.) At every vertex of the resulting tree, place an independent $\Po(\alpha)$ number of cars. There is only space for one of them, and any surplus cars drive towards the root, parking in the first available space.

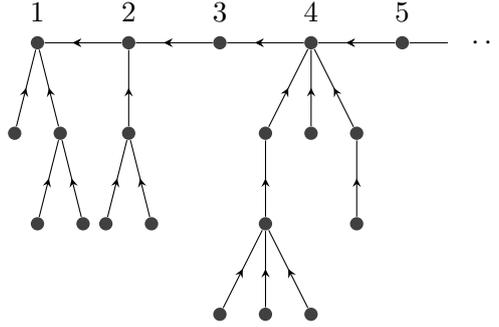
\begin{figure}[hbt]
\begin{tikzpicture}[scale=\pathscale,decoration={markings, mark=between positions 0 and 1 step \pathscale cm with {\arrow {stealth}}}]

    \foreach \pos/\name in {{(0,0)/1}, {(1,0)/2}, {(2,0)/3}, {(3,0)/4}, {(4,0)/5}} {
           \node[vertex] (v\name) at \pos {};
           \node[above = 0.4em] () at \pos {$\name$};
           }
        \foreach \pos/\name in {{(-0.5,0)/a1}, {(0.5,0)/a2}, {(1.5,0)/a3}, {(2.5,0)/a4}, {(3.5,0)/a5}}
           \node[] (\name) at \pos {};
	\foreach \source/\dest in {v2/v1, v3/v2, v4/v3, v5/v4}
         \draw [] (\source) -- (\dest);
	\foreach \source/\dest in {a2/a1, a3/a2, a4/a3, a5/a4}
         \draw [decorate] (\source) -- (\dest);
         
    \draw [] (4.5,0) -- (v5);
    \node[right = 2em] () at (v5) {$\cdots$};
    
    \node[vertex] (v11) at (-0.25,-1) {};
    \node[vertex] (v12) at (0.25,-1) {};
    \draw [] (v11) -- (v1);
    \draw [decorate] (v11)+(0.125,0.5) -- (v1);
    \draw [] (v12) -- (v1);
    \draw [decorate] (v12)+(-0.125,0.5) -- (v1);
    
    \node[vertex] (v121) at (0,-2) {};
    \node[vertex] (v122) at (0.5,-2) {};
    \draw [] (v12) -- (v121);
    \draw [decorate] (v121)+(0.125,0.5) -- (v12);
    \draw [] (v12) -- (v122);
    \draw [decorate] (v122)+(-0.125,0.5) -- (v12);
    
    \node[vertex] (v21) at (1,-1) {};
    \draw [] (v21) -- (v2);
    \draw [decorate] (v21)+(0,0.5) -- (v2);
    \node[vertex] (v211) at (0.75,-2) {};
    \node[vertex] (v212) at (1.25,-2) {};
    \draw [] (v21) -- (v211);
    \draw [decorate] (v211)+(0.125,0.5) -- (v21);
    \draw [] (v21) -- (v212);
    \draw [decorate] (v212)+(-0.125,0.5) -- (v21);
    
    \node[vertex] (v41) at (2.5,-1) {};
    \node[vertex] (v42) at (3,-1) {};
    \node[vertex] (v43) at (3.5,-1) {};
    \draw [] (v41) -- (v4);
    \draw [decorate] (v41)+(0.25,0.5) -- (v4);
    
    \node[vertex] (v411) at (2.5,-2) {};
    \draw [] (v41) -- (v411);
    \draw [decorate] (v411)+(0,0.5) -- (v41);
    \node[vertex] (v4111) at (2,-3) {};
    \node[vertex] (v4112) at (2.5,-3) {};
    \node[vertex] (v4113) at (3,-3) {};
    \draw [] (v411) -- (v4111);
    \draw [] (v411) -- (v4112);
    \draw [] (v411) -- (v4113);
    \draw [decorate] (v4111)+(0.25,0.5) -- (v411);
    \draw [decorate] (v4112)+(0,0.5) -- (v411);
    \draw [decorate] (v4113)+(-0.25,0.5) -- (v411);
    \draw [] (v42) -- (v4);
    \draw [decorate] (v42)+(0,0.5) -- (v4);
    \draw [] (v43) -- (v4);
    \draw [decorate] (v43)+(-0.25,0.5) -- (v4);
    \node[vertex] (v431) at (3.5,-2) {};
    \draw [] (v431) -- (v43);
    \draw [decorate] (v431)+(0,0.5) -- (v43);
    \foreach \source/\dest in {a2/a1, a3/a2, a4/a3, a5/a4}
      \draw [decorate] (\source) -- (\dest);
    
\end{tikzpicture}
\caption{The tree $T$, a critical Poisson--Galton--Watson tree conditioned on non-extinction. The trees attached to the path on $\NN$ are almost surely finite.}
\label{fig:treeLimit}
\end{figure}

\subsection{A local weak limit} 
Our model is the limit of the problem considered in \cite{LacknerPanholzer} in the sense of local weak convergence, which we now introduce.

First, let $\mathcal{G}$ be the set of graphs $G=(V(G),E(G))$ with finite or countably infinite vertex set $V(G)$ which are additionally \emph{locally finite} i.e.\ all vertex degrees are finite, which is equivalent to the property that for each $v \in V(G)$ and each $r \ge 0$, the number of vertices within graph distance $r$ of $v$ is finite.  Let $\mathcal{G}_* = \{(G,\rho): G \in \mathcal{G}, \rho \in V(G)\}$ be the set of rooted locally finite graphs, considered up to rooted isomorphism.  (We will abuse notation by writing $(G,\rho)$ for the equivalence class of $(G,\rho)$.)  For $(G,\rho) \in \mathcal{G}_*$, write $d_G$ for the graph distance in $G$, and let $B_G(\rho,r) = \{v \in V(G): d_G(\rho,v) \le r\}$, the (closed) ball of radius $r$ around $\rho$ in $G$. Write $G[\rho,r]$ for the induced subgraph of $G$.  We make $\mathcal{G}_*$ into a metric space by endowing it with the distance $d_{\mathrm{loc}}$ defined by
\[
d_{\mathrm{loc}}((G,\rho), (G', \rho')) = 2^{-\sup\{r \ge 0: G[\rho,r] \cong G'[\rho',r]\}}.
\]
Now let $(G,\rho)$ and $(G_n, \rho_n)_{n \ge 1}$ be random rooted locally finite graphs.  Then, following Benjamini and Schramm~\cite{BenjaminiSchramm} and Aldous and Steele~\cite{AldousSteele-LocalConvergence}, if $(G_n,\rho_n) \convdist (G,\rho)$ with respect to this topology, we say that $(G,\rho)$ is the \emph{local weak limit} of $(G_n, \rho_n)_{n \ge 1}$. It is a well-known fact, first observed by Grimmett~\cite{Grimmett-TreeLimit}, that $(T, \rho)$ (with $\rho=1$) is the local weak limit of $(T_n, \rho_n)_{n \ge 1}$, where $\rho_n$ is the progenitor of the branching process.  Note, in particular, that $(T,\rho)$ is locally finite. (Indeed, it has quadratic volume growth, in the sense that there exists a constant $C > 0$ such that
\[
\Prob{|B_T(\rho,r)| > \lambda r^2} \le C\exp(-C\lambda), \quad \lambda \ge 0.
\]
This is essentially a consequence of Proposition 2.7 of Barlow and Kumagai~\cite{BarlowKumagai}; see the discussion in Section 5.3 of Addario-Berry~\cite{Addario-Berry}.) 

Now, for each $v \in V(T_n)$, let $P_{n,m}(v)$ be the number of cars wishing to park at $v$ out of the total of $m$ cars. The vector $(P_{n,m}(v), v \in V(T_n))$ has a $\mathrm{Multinomial}(m; 1/n, \ldots, 1/n)$ distribution and so, for any finite subset $S \subset V(T_n)$ which is chosen independently of $(P_{n,m}(v), v \in V(T_n))$, 
\[
(P_{n,\fl{\alpha n}}(v), v \in S) \convdist (P(v), v \in S),
\]
where the random variables $(P(v), v \in S)$ are i.i.d.\ $\Po(\alpha)$.  

In order to combine these results, we treat the numbers of cars as integer-valued marks on the vertices of our trees.  Let $\mathcal{M} = \{(G,\rho, \mathbf{x}): (G,\rho) \in \mathcal{G}_*, \mathbf{x} \in \{0,1,2,\ldots\}^{V(G)}\}$, the space of marked locally finite rooted graphs.  For $(G,\rho,\mathbf{x}), (G',\rho',\mathbf{x}') \in \mathcal{M}$, let $R((G,\rho, \mathbf{x}), (G',\rho',\mathbf{x}'))$ be the supremum of the set of $r \ge 0$ such that there exists an isomorphism $\phi: V(B_G(\rho,r)) \to V(B_{G'}(\rho',r))$ of $G[\rho,r]$ and $G'[\rho',r]$ such that additionally $x_v = x'_{\phi(v)}$ for all $v \in B_G(\rho,r)$.  Then letting $d_{\mathcal{M}}((G,\rho, \mathbf{x}), (G',\rho',\mathbf{x}')) = 2^{-R((G,\rho, \mathbf{x}), (G',\rho',\mathbf{x}'))}$ it is straightforward to verify that $(\mathcal{M}, d_{\mathcal{M}})$ is a Polish space.  With respect to the induced topology, we obtain
\begin{equation} \label{eqn:locallimit}
(T_n, \rho_n, (P_{n,\fl{\alpha n}}(v), v \in V(T_n))) \convdist (T, \rho, (P(v), v \in V(T)))
\end{equation}
as $n \to \infty$, where $(P(v), v \in V(T))$ are i.i.d.\ $\Po(\alpha)$ random variables depending on $T$ only through its vertex-labels.

\subsection{Main results}

The main part of our investigation of parking on random trees will be analysing the process on a PGW(1) tree. We summarise our results in the following theorem. (We will discuss the definition and properties of the Lambert W-function in Section \ref{sec:PoiGWcritical}.)
\begin{thm}
\label{thm:PoiGWcritical}
 Let $X$ denote the number of cars that visit the root of a $\PGW(1)$ tree with, for some $\alpha \in (0,1)$, an independent $\Po(\alpha)$ number of cars initially picking every vertex.
 \begin{enumerate}
  \item \label{thm:PoiGWcritical:smallAlpha}
  If $\alpha \in (0,1/2]$ then the probability generating function of $X$ is
   \[
    G(s) = -s W_{-1} \left(-\frac{1}{s} \exp \left(\alpha s - \alpha - 1 + (1-s^{-1})(1-\alpha) \right) \right),
   \]
   where $W_{-1}(x)$ is the $(-1)$-th branch of the Lambert W-function. Consequently, we have $p = \Prob{X=0} = 1-\alpha$ and $\E{X} = 1 - \sqrt{1 - 2\alpha}$.  
  \item
  \label{thm:PoiGWcritical:largeAlpha}
  If $\alpha > 1/2$ then we have $p = \bP(X=0) \in (1-\alpha,\frac{1}{4\alpha})$ and, taking
  \[
   s_p = \frac{1-\sqrt{1-4 p \alpha}}{2 \alpha},
  \]
  $p$ satisfies
  \[
   s_p^{-1} \exp \left (\alpha s_p - \alpha + \left ( 1-s_p^{-1} \right ) p \right ) - 1 = 0.
  \]
 Moreover, the probability generating function of $X$ is
  \[
    G(s) = -s W_{i} \left(-\frac{1}{s} \exp \left(\alpha s - \alpha - 1 + (1-s^{-1})p \right) \right),
  \]
  where $i = -1$ for $s \leq s_p$ and $i = 0$ otherwise. Consequently, for $\alpha > 1/2$ we have $\E{X} = \infty$.
 \end{enumerate}
\end{thm}

Perhaps the most striking aspect of Theorem~\ref{thm:PoiGWcritical} is that the quantity $\E{X}$ undergoes a discontinuous phase transition at $\alpha = 1/2$:
\begin{equation} \label{eqn:discphtr}
\E{X} =  \begin{cases}
1 - \sqrt{1-2\alpha} & \text{ for $\alpha \le 1/2$} \\
\infty & \text{ for $\alpha > 1/2$.}
\end{cases}
\end{equation}
We will discuss this phenomenon further in Section \ref{sec:generalisations}.

The second main result of this paper, which to a large extent is a corollary of Theorem \ref{thm:PoiGWcritical}, is the following theorem about parking on $T$.
\begin{thm}
\label{thm:infiniteParkingProb}
Let $T$ be a $\PGWinf(1)$ tree, rooted at $\rho$, with all edges directed towards $\rho$. Assume that an independent $\Po(\alpha)$ number of cars arrives at each vertex of the tree. Let $A_{\alpha}$ be the event that all the cars can park on $T$. Then
\[
\bP( A_\alpha) = \begin{cases}
\frac{\sqrt{1-2\alpha}}{1-\alpha} & \text{ if $0 \le \alpha \le 1/2$,} \\
0 & \text{ if $\alpha > 1/2$.}
\end{cases}
\]
\end{thm}
In particular, we recover the phase transition and limiting probabilities of Theorem~\ref{thm:L&P}.

We analyse the process of parking on $T$ in two stages. In the first stage, we limit our attention to the process on the critical Galton--Watson trees attached to the path $\Pi_{\infty}$. Our aim is to understand the random number of cars that visit the root of such a subtree, either because they initially chose to park there or because they have traversed the whole path from some other vertex of the subtree (we think of these cars as stopping at the root of their subtree and waiting till the end of the first stage).  We denote this random number of cars by $X$. The recursive definition of Galton--Watson trees allows us to express $X$ as a solution to the following recursive distributional equation (RDE):
\begin{equation} \label{eqn:rde}
X \equidist P + \sum_{i=1}^N (X_i -1)^+,
\end{equation}
where $P \sim \Po(\alpha)$, $N \sim \Po(1)$, $X_1, X_2, \ldots$ are i.i.d.\ copies of the (non-negative integer-valued) random variable $X$, $(X_i-1)^+ = \max\{X_i-1,0\}$, and all of the random variables on the right-hand side are independent.  (See the survey paper of Aldous and Bandyopadhyay~\cite{AldousBandyopadhyay} for more on the theory of RDE's.) Since the critical Galton--Watson tree is finite almost surely, and $X$ gives an explicit construction of a solution to \eqref{eqn:rde}, we obtain both existence and uniqueness of $X$. We use generating functions to understand the distribution of this solution and obtain the expressions in Theorem~\ref{thm:PoiGWcritical}.

Once we understand the law of $X$, we look at the parking process on the path $\Pi_\infty$ with $X_i$ cars arriving at $i \in \NN$, where $X_1, X_2, \ldots$ are i.i.d.\ copies of $X$.  The crucial observation here is that the cars can all park on $\Pi_\infty$ if and only if we have
\[
C_n = n - \sum_{k=1}^n X_k \geq 0 \quad \text{ for all $n \in \NN$.}
\]
This is because the first $n$ vertices of the path provide us with $n$ parking places, and the number of cars wishing to park in these spaces is at least $\sum_{k=1}^n X_k$: hence if $C_n$ is negative for some $n$ then we do not have a parking function for $\Pi_\infty$. On the other hand, if we do not have a parking function for $\Pi_\infty$ then there is some smallest $n$ such that the cars starting their journey on $[n]$ cannot all park on that initial segment of the path, and so we must have $C_n < 0$.

It will be useful to us later to know exactly how many cars arrive at 1.  $C_n$ is the difference between the total number of cars arriving somewhere in $\{1,2,\ldots, n\}$ and the number of available spaces.  If $C_n$ is negative then there is insufficient space to accommodate all of the cars arriving in $\{1,2,\ldots,n\}$ and at least $X_1 + (X_2 - 1) + \cdots + (X_n - 1) = 1 - C_n$ wish to park at $1$ (``at least'' because it may be that spare capacity comes after it is needed and so, in fact, more cars wish to park at the root).  If $(C_n)_{n \ge 1}$ attains a new minimum at some $m$ then all of the vertices labelled $1, 2, \ldots, m$ must be occupied by a car, and so \emph{exactly} $1-C_m$ cars eventually arrive at $1$ from somewhere in $\{1,2,\ldots,m\}$.  It follows that the number which visit 1 is $1 - \inf_{n \ge 1} C_n$.

Another useful observation will be that $X$ is stochastically increasing in $\alpha$, since if $\alpha < \alpha'$ then we may couple the Poisson numbers of cars $P_v^{(\alpha)}$ and $P_v^{(\alpha')}$ wanting to park at each vertex $v$ in such a way that $P^{(\alpha')}_v \ge P^{(\alpha)}_v$.  It is then easy to see that the number of cars wanting to park at the root must be larger for $\alpha'$.  

Let us now show how Theorem \ref{thm:infiniteParkingProb} follows from Theorem \ref{thm:PoiGWcritical}.
\begin{proofOfTheorem}{thm:infiniteParkingProb}
The process $(C_n)_{n \ge 1}$ is a random walk with initial state $C_0 = 0$ and step-size $1-X_n$ for $n=1,2,\ldots$ The asymptotic behaviour of $(C_n)$ depends entirely on its mean.  Indeed,
\[
\Prob{C_n \geq 0 \text{ for all } n \geq 1} > 0
\]
if and only if $\E{1-X} > 0$, i.e., if and only if $\E{X} < 1$. By Theorem \ref{thm:PoiGWcritical} we see that this occurs if and only if $\alpha < 1/2$.  In that case, $(C_n)_{n \ge 1}$ is a random walk with positive drift which is \emph{skip-free to the right}, i.e., a random walk with
\[
 \E{C_{n+1}-C_n}>0 \quad \text{and} \quad \Prob{C_{n+1}-C_n \geq 2} = 0.                                                                                                                                                 
\]
This enables a particularly convenient calculation of its hitting probabilities.  We obtain (see, e.g., Brown, Pek\"oz and Ross \cite{BrownPekozRoss-skipFreeWalks})
\begin{equation}
\label{eq:skipfree}
\Prob{C_n \ge 0 \text{ for all } n \ge 1} = \frac{\E{C_2-C_1}}{\Prob{C_2-C_1=1}} = \frac{1-\E{X}}{\Prob{X = 0}}.
\end{equation}
 Theorem \ref{thm:infiniteParkingProb} now follows trivially from \eqref{eq:skipfree} since, by Theorem \ref{thm:PoiGWcritical} case (\ref{thm:PoiGWcritical:smallAlpha}), for all $\alpha \in (0,1/2)$ we have
 \[
 \Prob{A_{\alpha}} =  \Prob{C_n \ge 0 \text{ for all } n \ge 1} = \frac{\sqrt{1 - 2\alpha}}{1-\alpha},
 \]
 while for $\alpha \geq 1/2$, by stochastic monotonicity in $\alpha$ we obtain
 \[
\Prob{A_{\alpha}} =   \Prob{C_n \ge 0 \text{ for all } n \ge 1} \leq \inf_{\alpha \in (0,1/2)} \frac{\sqrt{1 - 2\alpha}}{1-\alpha} = 0. \qedhere
 \]
\end{proofOfTheorem}

Having analysed the local weak limit, it remains to prove that the probability that all cars can park behaves continuously with respect to this notion of convergence.

\begin{proofOfTheorem}{thm:L&P}
For an arbitrary rooted tree $(\tau, \rho)$ and arbitrary numbers $\pi = (\pi(v), v \in V(\tau))$ of arrivals at its vertices, write $\chi(\tau,\pi)$ for the number of cars arriving at the root.  We begin by observing the simple fact that $\chi$ is monotone in both of its arguments: 
\begin{itemize}
\item if $\pi(v) \le \pi'(v)$ for all $v \in V(\tau)$ then $\chi(\tau, \pi) \le \chi(\tau, \pi')$;
\item if $\tau$ is a subtree of $\tau'$ (with the same root) and $\pi'$ gives the numbers of arrivals in $\tau'$ then $\chi(\tau, \pi'|_{v \in V(\tau)}) \le \chi(\tau', \pi')$.
\end{itemize}
We wish to prove that
\[
\lim_{n \to \infty} \Prob{A_{n,\alpha}} = \Prob{A_{\alpha}},
\]
where
\[
A_{n,\alpha} = \left\{\chi(T_n, P_{n,\fl{\alpha n}}) \in \{0,1\}\right\} \quad \text{and} \quad A_{\alpha} = \left\{\chi(T, P) \in \{0,1\}\right\}.
\]

First observe that Theorem 4.1 of Luczak and Winkler~\cite{LuczakWinkler} entails that there exists a coupling of the trees $(T_n)_{n \ge 1}$ which is increasing.  (See the discussion below Theorem 2.1 of Lyons, Peled and Schramm~\cite{LyonsPeledSchramm} for how to deduce this from \cite{LuczakWinkler}.)  Let us use this coupling, and take $T$ to be its increasing limit.  For notational simplicity, when convenient we will label the vertices of $T$ by $\NN$, with the vertex labelled $n$ being the vertex which appears for the first time in $T_n$. (Observe that this is \emph{not} the labelling by $[n]$ which makes $T_n$ a uniform labelled tree.)

We now turn to the arrivals processes of cars.  Given $\beta > 0$, let $(P^{(\beta)}(i), i \in \NN)$ be independent and identically distributed Po($\beta$) random variables, independent of $T$, so that
\[
(P(i), i \in \NN) \equidist (P^{(\alpha)}(i), i \in \NN).
\]
We will make use of the following well-known fact about the Poisson distribution: for any $\beta > 0$, conditional on $\sum_{i=1}^n P^{(\beta)}(i) = m$, the joint distribution of $(P^{(\beta)}(1), \ldots, P^{(\beta)}(n))$ is $\text{Multinomial}(m; 1/n, \ldots, 1/n)$.  Indeed, observe that we may realise $P^{(\beta)}(1), \ldots, P^{(\beta)}(n)$ by taking a Poisson point process of intensity $\beta$ on $\R_+$ and taking $P^{(\beta)}(i)$ to be the number of points falling in the interval $(i-1,i]$ for $1 \le i \le n$.  Given the point configuration, suppose that we remove $\left(\sum_{i=1}^n P^{(\beta)}(i) - m\right)^+$ of the points, chosen independently and uniformly at random.  Write $P'(i)$ for the number of remaining points in $(i-1,i]$, for $1 \le i \le n$.  Then on the event
\[
\left\{ \sum_{i=1}^n P^{(\beta)}(i) \ge m\right\},
\]
we have $(P'(1), \ldots, P'(n)) \sim \text{Multinomial}(m;1/n, \ldots, 1/n)$.

\textbf{Case $\alpha < 1/2$, lower bound.}  Let $\beta$ be such that $\alpha < \beta < 1/2$. Let
\[
E'_n = \left\{ \sum_{i=1}^n P^{(\beta)}(i) \ge \fl{\alpha n} \right\}
\]
and note that, by the weak law of large numbers, $\frac{1}{n} \sum_{i=1}^n P^{(\beta)}(i) \convprob \beta$, so that $\Prob{E'_n} \to 1$ as $n \to \infty$.  Initially allocate $P^{(\beta)}(i)$ cars to vertex $i \in \NN$.  Remove $\left(\sum_{i=1}^n P^{(\beta)}(i) - \fl{\alpha n} \right)^+$ cars chosen uniformly at random from among those on vertices in $[n]$, and write $P'_{n,\fl{\alpha n}}(i)$ for the resulting numbers of cars at vertex $i$ for $i \in [n]$.  We clearly have $P'_{n, \fl{\alpha n}} \le P^{(\beta)}(i)$ for all $i \in [n]$.  Moreover, on the event $E'_n$,
\[
\left(P'_{n,\fl{\alpha n}}(i), i \in [n]\right) \equidist \left(P_{n,\fl{\alpha n}}(i), i \in [n]\right).
\]
Hence, on $E'_n$ we have
\[
\chi(T_n, P'_{n,\fl{\alpha n}}) \le \chi(T,P^{(\beta)}).
\]
So for all $n \ge 1$,
\[
\Prob{\chi(T_n, P'_{n,\fl{\alpha n}}) \in \{0,1\}} \ge \Prob{\left\{ \chi(T,P^{(\beta)}) \in \{0,1\} \right\} \cap E'_n} 
\]
and hence
\begin{equation} \label{eqn:lowerbound}
\liminf_{n \to \infty}\Prob{\chi(T_n, P'_{n,\fl{\alpha n}}) \in \{0,1\}} \ge \frac{\sqrt{1-2\beta}}{1-\beta}.
\end{equation}

\textbf{Case $\alpha < 1/2$, upper bound}.  Let $\gamma$ be such that $0 < \gamma < \alpha < 1/2$. We perform an analogous coupling of the arrivals: let
\[
E_n'' = \left\{ \sum_{i=1}^n P^{(\gamma)}(i) \le \fl{\alpha n} \right\}
\]
and note that given $\epsilon > 0$, there exists $n_{\epsilon}$ such that for all $n \ge n_{\epsilon}$ we have $\Prob{E_n''} >1-\epsilon/3$.  Initially allocate $P^{(\gamma)}(i)$ cars to vertex $i \in \NN$.  Add $\left(\fl{\alpha n} - \sum_{i=1}^n P^{(\gamma)}(i) \right)^+$ cars to independent and uniformly chosen vertices in $[n]$ and write $P''_{n,\fl{\alpha n}}(i)$ for the resulting numbers of cars at vertex $i$ for $i \in [n]$.  Clearly we have $P''_{n,\fl{\alpha n}}(i) \ge P^{(\gamma)}(i)$ for all $i \in [n]$.  On the event $E''_n$,
\[
\left(P''_{n,\fl{\alpha n}}(i), i \in [n]\right) \equidist \left(P_{n,\fl{\alpha n}}(i), i \in [n]\right).
\]

Now note that
\[
\chi(T,P^{(\gamma)}|_{B_T(\rho,r)}) \uparrow \chi(T,P^{(\gamma)})
\]
as $r \to \infty$.  Recall the random walk representation for parking on $T$.  We have $\chi(T,P^{(\gamma)}) \equidist 1- \inf_{n \ge 1} C_n$. Since $\gamma < 1/2$, the random walk has positive drift and so $\chi(T,P^{(\gamma)}) < \infty$ almost surely.  Hence, given $\epsilon > 0$, there exists $r_{\epsilon}$ such that for all $r \ge r_{\epsilon}$, we have
\[
\Prob{\chi(T,P^{(\gamma)}|_{B_T(\rho,r)}) \neq \chi(T,P^{(\gamma)})} < \epsilon/3.
\]
Moreover, there exists $n_{\epsilon, r}$ such that for all $n \ge n_{\epsilon,r}$,
\[
\Prob{B_{T}(\rho,r) \neq B_{T_n}(\rho_n,r)} < \epsilon/3.
\]
On the event $\{\chi(T,P^{(\gamma)}|_{B_T(\rho,r)}) = \chi(T,P^{(\gamma)})\} \cap \{B_{T}(\rho,r) = B_{T_n}(\rho_n,r)\} \cap E_n''$, we have
\[
\chi(T,P^{(\gamma)}) = \chi(T, P^{(\gamma)}|_{B_T(\rho,r)}) \le \chi(T_n, P''_{n, \fl{\alpha n}}|_{B_{T_n}(\rho_n,r)}) \le \chi(T_n,P''_{n, \fl{\alpha n}}).
\]
Hence, for $n \ge \max\{n_{\epsilon}, n_{\epsilon,r_{\epsilon}}\}$,
\begin{align*}
& \Prob{\chi(T_n,P''_{n, \fl{\alpha n}}) \in \{0,1\}} \\
&\qquad  \le \Prob{\chi(T,P^{(\gamma)}) \in \{0,1\}} \\
& \qquad \qquad +  \Prob{(E_n'')^c} + \Prob{\chi(T,P^{(\gamma)}|_{B_T(\rho,r_{\epsilon})}) \neq \chi(T,P^{(\gamma)})} + \Prob{B_{T}(\rho,r_{\epsilon}) \neq B_{T_n}(\rho_n,r_{\epsilon})}\\
& \qquad < \frac{\sqrt{1-2\gamma}}{1-\gamma} + \epsilon.
\end{align*}
But $\epsilon > 0$ was arbitrary and so
\begin{equation} \label{eqn:upperbound}
\limsup_{n \to \infty} \Prob{\chi(T_n,P''_{n, \fl{\alpha n}}) \in \{0,1\}} \le \frac{\sqrt{1-2\gamma}}{1-\gamma}.
\end{equation}
\textbf{Case $\alpha < 1/2$}. Now recall that $\gamma$ and $\beta$ were chosen arbitrarily such that $\gamma < \alpha < \beta$.  Using (\ref{eqn:lowerbound}), (\ref{eqn:upperbound}) and the fact that the function $x \mapsto \frac{\sqrt{1-2x}}{1-x}$ is continuous on $(0,1/2]$ with value 0 at $x=1/2$, we obtain
\[
\lim_{n \to \infty} \Prob{A_{n,\alpha}} = \frac{\sqrt{1-2\alpha}}{1-\alpha}
\]
for $\alpha < 1/2$.

\textbf{Case $\alpha \ge 1/2$}.  This follows straightforwardly since, by coupling, for $\alpha \ge 1/2$ we have
\[
\lim_{n \to \infty} \Prob{\chi(T_n, P_{n,\fl{\alpha n}}) \in \{0,1\}} \le \inf_{\gamma < 1/2} \ \lim_{n \to \infty} \Prob{\chi(T_n, P_{n, \fl{\gamma n}}) \in \{0,1\})} = 0. \qedhere
\]
\end{proofOfTheorem}

The rest of the paper is organised as follows. In Section \ref{sec:PoiGWcritical} we prove Theorem \ref{thm:PoiGWcritical}, which is now the only missing piece in our proof of Theorem~\ref{thm:L&P}.  In Section \ref{sec:generalisations} we discuss some generalisations of our results.  In particular, we discuss a related model studied by Jones~\cite{Jones}.

\section{Parking on a critical Poisson Galton--Watson tree}
\label{sec:PoiGWcritical}

The following simple proposition gives us a first piece of information about parking on critical Galton--Watson trees.
\begin{prop}
\label{prop:p0}
Let $\alpha \in (0,1)$ and let $X$ denote the number of cars that arrive at the root of a critical Galton--Watson tree with $\Po(1)$ offspring distribution. We have
\[
 p = \Prob{X=0} \geq \exp(-1-\alpha) > 0.
\]
Moreover, if the solution to the RDE \eqref{eqn:rde} has a finite mean then $p = 1-\alpha$.
\end{prop}
\begin{proof}
 The lower bound on $p$ follows from the fact that if the root of the Galton--Watson tree has zero children and no cars want to park at it directly then we have $X=0$. Thus
 \[
  p \geq \Prob{N=0, P=0} = \exp(-1)\exp(-\alpha).
 \]
 Now, taking expectations in \eqref{eqn:rde}, we obtain
 \[
  \E{X} = \alpha + \E{X} - \Prob{X \ge 1}
 \]
 so that either $\Prob{X \ge 1} = \alpha$ or $\E{X} = \infty$.
\end{proof}

Let $G(s) = \E{s^X}$, $s \ge 0$, be the probability generating function of $X$. We have
\begin{align}
\label{eq:functionalEquation}
G(s) & = \E{s^P} \E{\E{s^{(X - 1)^+}}^N } \notag \\
& = \exp(\alpha(s-1)) \exp \left(\E{s^{(X-1)^+}} -1 \right) \notag \\
& = \exp(\alpha(s-1) - 1) \exp \left(\E{s^{X-1}} + (1-s^{-1})p \right) \notag \\
& = \exp \left(s^{-1}G(s) + \alpha s - \alpha - 1 + (1 - s^{-1})p \right).
\end{align}

The aim of the lemmas that follow is to show that for $\alpha \leq 1/2$ we indeed have $p = 1-\alpha$, i.e., the value suggested by Proposition \ref{prop:p0}.
\begin{lemma}
\label{lem:pNotTooSmall}
 For any $\alpha \in (0,1)$, we have $p \geq 1-\alpha$.
\end{lemma}
\begin{proof}
 Our proof is based on the calculation of the expectation of $X$. To find $\E{X}$ we use Abel's Theorem, which states that $\E{X} = G'(1-)$. Differentiating \eqref{eq:functionalEquation}, we obtain
\[
G'(s) = [-s^{-2} G(s) + s^{-1} G'(s) + \alpha + p s^{-2}] G(s)
\]
and rearranging yields
\begin{equation}
\label{eq:expectationIntermediate}
G'(s) = \frac{(\alpha s^2  + p - G(s))G(s)}{s(s-G(s))}.
\end{equation}
Recall that $X < \infty$ almost surely, so that $G(1) = 1$.  So as $s \to 1$, the limit of the denominator in \eqref{eq:expectationIntermediate} is $0$. If $p < 1-\alpha$, the limit of the numerator is some negative constant. Hence the expectation of $X$ is infinite in absolute value, and since $\E{X} = -\infty$ is impossible, we must have that $G(s)-s$ converges to zero from above. But since $G(s) \le 1$ for $s \in [0,1]$, this implies that, as $s \to 1$, the limit of the derivative of $G(s)$ is at most $1$ i.e.\ $\E{X} \le 1$, contradicting $\E{X} = \infty$. Hence we must have $p \geq 1-\alpha$.
\end{proof}

It remains to show that $p \leq 1-\alpha$ when $\alpha \le 1/2$. This turns out to be more complicated and we need to learn more about the exact form of $G(s)$ in order to achieve it.

Let $W_i$, $i \in \Z$, denote the branches of the Lambert W-function, i.e.\ the branches of the inverse of $f(z) = ze^z$, $z \in \mathbb{C}$. In particular, this implies that for all $i \in \Z$ we have $W_i(z) e^{W_i(z)} = z$. (See, for example, Corless, Gonnet, Hare, Jeffrey and Knuth~\cite{CGHJK-Lambert}.) Recall that
\[
 W_{-1}:[-e^{-1},0) \to (-\infty, -1] \ \mbox{ and } \ W_0:[-e^{-1},\infty) \to (-1,\infty]
\]
are the two real-valued branches of $W$. We shall often use the following property of the Lambert W-function.
\begin{fact}
\label{fact:ex^x}
 For all $x \leq -1$ we have $W_{-1} (x e^x) = x$. 
\end{fact}
\begin{proof}
 Let $x < -1$. Obviously, taking $y=x$ we obtain a solution to $ye^y = xe^x$, hence there is some branch $W_i$ of the Lambert W-function such that $W_{i} (x e^x) = x$. Since $x \in \R$, we must have $i=0$ or $i=-1$. However, we know that $W_0(x) > -1$ for all $x \geq -e^{-1}$, so we must have $W_{-1} (x e^x) = x$. We complete the proof of the fact by observing that also $W_{-1} (- e^{-1}) = -1$.
\end{proof}

In the following lemma we show that there are only two possible values that $G(s)$ can take for any $s \in (0,1)$.
\begin{lemma}
 \label{lem:mustBeTheW}
 For all $s \in (0,1]$ we have 
 \begin{equation}
 \label{eq:GwithSomeW}
  G(s) = f_i(s) = -s W_i \left(-\frac{1}{s} \exp \left(\alpha s - \alpha - 1 + (1-s^{-1})p \right) \right)
 \end{equation}
 for some $i = i(s) \in \{0,-1\}$.
\end{lemma}
\begin{proof}
 Multiplying both sides of \eqref{eq:functionalEquation} by $-s^{-1} \exp \left( -s^{-1}G(s) \right)$ we obtain
 \[
  -s^{-1}G(s) \exp \left( -s^{-1}G(s) \right) = - s^{-1} \exp \left(\alpha s - \alpha - 1 + (1 - s^{-1})p \right).
 \]
By the definition of the Lambert W-function, this implies that
 \[
  -s^{-1}G(s) = W_k \left (- \frac{1}{s} \exp \left(\alpha s - \alpha - 1 + (1 - s^{-1})p \right) \right)
 \]
 for some $k \in \Z$. The lemma then follows from the fact that $G(s)$ must take real values.
\end{proof}

The condition that $G(0) = p > 0$ and the continuity of $G$ allow us to identify that for all $\alpha \in (0,1)$, $G(s) = f_{-1}(s)$ in a neighbourhood of $s=0$.
\begin{lemma}
 \label{lem:W-1AtS=0}
 For all $\alpha \in (0,1)$ there exists some $\varepsilon_\alpha >0$ such that for $s \in (0,\varepsilon_\alpha)$ we have 
 \[
  G(s) = -s W_{-1} \left(-\frac{1}{s} \exp \left(\alpha s - \alpha - 1 + (1-s^{-1})p \right) \right).
 \]
\end{lemma}
\begin{proof}
 To prove the lemma it is enough to show that $\lim_{s \to 0} f_0(s) = 0 \neq p = G(0)$. Indeed, since $p > 0$, we have
 \[
  -\frac{1}{s} \exp \left(\alpha s - \alpha - 1 + (1-s^{-1})p \right) \to 0
 \]
 as $s \to 0$. Since $W_0$ is continuous and satisfies $W_0(0)=0$, this implies $\lim_{s \to 0} f_0(s) = 0$.
\end{proof}
As a check, we observe that $W_{-1}(x) \sim \log(-x)$ for $x \uparrow 0$, and so as $s \downarrow 0$ we have
\[
-s W_{-1} \left(-\frac{1}{s} \exp \left(\alpha s - \alpha - 1 + (1-s^{-1})p \right) \right) \to p.
\]

Both $W_0(s)$ and $W_{-1}(s)$ are defined on $[-e^{-1},\infty)$ and they are equal if and only if $s=-e^{-1}$. For $\alpha \in (0,1/2]$ and $p \geq 1-\alpha$ this allows us to identify $W_{-1}$ as the branch of the Lambert W-function that gives us the formula for $G(s)$ for all $s \in (0,1]$.
\begin{cor}
\label{cor:W_-1forSmallAlpha}
 If $\alpha \leq 1/2$ then
  \begin{equation}
 \label{eq:W-1forSmallAlpha}
  G(s) = -s W_{-1} \left(-\frac{1}{s} \exp \left(\alpha s - \alpha - 1 + (1-s^{-1}) p \right) \right).
 \end{equation}
 for all $s \in (0,1]$.
\end{cor}
\begin{proof}
 By Lemma \ref{lem:W-1AtS=0}, the corollary holds in some small neighbourhood of $0$. By the continuity of $G(s)$ and of the branches of the W-function, in order to complete the proof it is therefore enough to show that $f_0(s) \neq f_{-1}(s)$ for all $s \in (0,1)$.
 
 To do this, we first observe that the argument of $W$ in \eqref{eq:W-1forSmallAlpha} equals $-e^{-1}$ for $s=1$, so consequently $f_0(1) = f_{-1}(1)$. The corollary will follow if we can show that for all $s \in (0,1)$ we have
 \[
  -\frac{1}{s} \exp \left(\alpha s - \alpha - 1 + (1-s^{-1}) p \right) > -\exp(-1),
 \]
 which is equivalent to
 \[
  g(s) = \alpha s - \alpha + (1-s^{-1}) p < \log s.
 \]
 Since $g(1) = \log(1) = 0$, this will, in turn, follow if $g'(s) > 1/s$ for all $s \in (0,1)$. We have $g'(s) > 1/s$ if
 \[
  \alpha s^2 - s + p > 0.
 \]
 Now, recalling that by Lemma \ref{lem:pNotTooSmall} we have $p \ge 1-\alpha$, we obtain
 \[
\alpha s^2 - s + p \ge \alpha s^2 - s + 1-\alpha = \alpha(s-1) \left (s-\frac{1}{\alpha}+1 \right ),
 \]
and the right-hand side is strictly positive for all $s \in (0,1)$ if $\alpha \leq 1/2$. So we do indeed have $g'(s) > 1/s$ for all $s \in (0,1)$. Hence, for $\alpha \leq 1/2$ the graphs of $f_0(s)$ and $f_{-1}(s)$ do not intersect in $(0,1)$, and since $f_{-1}(s)$ gives the formula for $G(s)$ near $0$, the corollary follows.
\end{proof}

\begin{cor}
 \label{cor:psmallAlpha}
 For all $\alpha \in (0,1/2]$, we have $p = 1-\alpha$.
\end{cor}
\begin{proof}
 By Corollary \ref{cor:W_-1forSmallAlpha} we have $G(s) = f_{-1}(s)$ for all $s \in (0,1]$. Suppose that $p > 1-\alpha$. Then $s^* = (1-p)/\alpha \in (0,1)$ and so $-1/s^* < -1$. Since also
 \begin{align*}
  \alpha s^* - \alpha - 1 + \left (1- \frac{1}{s^*} \right ) p & = 1-p - \alpha - 1 + \frac{1-p-\alpha}{1-p}p \\
   & = \frac{-p-\alpha+p^2+\alpha p + p - p^2-\alpha p}{1-p} \\
   & = \frac{-\alpha}{1-p} = -\frac{1}{s^*},
 \end{align*}
by plugging $s=s^*$ into \eqref{eq:W-1forSmallAlpha} by Fact \ref{fact:ex^x} we obtain $G(s^*)=1$. This is a contradiction since we do not have $\Prob{X=0} = 1$. Hence we must have $p = 1-\alpha$.
\end{proof}

Once we know that for $\alpha \leq 1/2$ we have $p = 1-\alpha$, we can also find $\E{X}$.
\begin{lemma}
\label{lem:E[X]smallAlpha}
 For $\alpha \in (0,1/2]$, we have $\E{X} = 1 - \sqrt{1 - 2\alpha}$.
\end{lemma}
\begin{proof}
 By \eqref{eq:expectationIntermediate} and Corollary \ref{cor:psmallAlpha} we have
 \[
  G'(s) = \frac{(\alpha s^2  + 1-\alpha - G(s))G(s)}{s(s-G(s))}.
 \]
 Since both numerator and denominator tend to $0$ as $s \uparrow 1$, we apply L'H\^opital's rule to see that
\[
\lim_{s \uparrow 1} \frac{\alpha s^2  + 1 - \alpha - G(s)}{s-G(s)} = \lim_{s \uparrow 1} \frac{2 \alpha s - G'(s)}{1 - G'(s)} = \frac{2 \alpha - G'(1-)}{1 - G'(1-)},
\]
which gives the relation
\[
G'(1-) = \frac{2 \alpha - G'(1-)}{1 - G'(1-)}.
\]
Rearranging, we obtain
\[
G'(1-)^2 - 2 G'(1-) + 2 \alpha = 0
\]
and so $G'(1-) = 1 \pm \sqrt{1 - 2\alpha}$.  Since $X$ is stochastically increasing in $\alpha$, we have that $\E{X}$ is an increasing function of $\alpha$. So this identifies $\E{X} = 1 - \sqrt{1 - 2\alpha}$.
\end{proof}
Equipped with Lemma~\ref{lem:E[X]smallAlpha} and Corollary \ref{cor:psmallAlpha} we can also deduce that $\E{X} = \infty$ when $\alpha > 1/2$.
\begin{cor}
 \label{cor:E[X]largeAlpha}
 For $\alpha > 1/2$ we have $\E{X} = \infty$.
\end{cor}
\begin{proof}
 Obviously $\E{X}$ is either a positive real constant or $\infty$. By the same argument as in the proof of Lemma~\ref{lem:E[X]smallAlpha} we see that if $p = 1-\alpha$ then $G'(1)$ is either infinite in absolute value or complex, and so $\E{X}$ must be $\infty$. If however $p \neq 1-\alpha$ then by Proposition \ref{prop:p0} we again have $\E{X} = \infty$.
\end{proof}

Theorem \ref{thm:PoiGWcritical} case (\ref{thm:PoiGWcritical:smallAlpha}) now follows immediately from Corollary \ref{cor:psmallAlpha}, Lemma~\ref{lem:E[X]smallAlpha} and Corollary \ref{cor:W_-1forSmallAlpha}, and Theorem \ref{thm:PoiGWcritical} case (\ref{thm:PoiGWcritical:largeAlpha}) is Corollary \ref{cor:E[X]largeAlpha}.

Before moving on to the proof of Theorem \ref{thm:infiniteParkingProb}, let us discuss the case $\alpha > 1/2$ a bit further. We shall find this useful in Section \ref{sec:generalisations} where we look at other related models.

We first show that if $\alpha > 1/2$ then we have $p > 1-\alpha$ (note that by Proposition \ref{prop:p0} this also implies that $\E{X} = \infty$ for $\alpha > 1/2$).
\begin{lemma}
\label{lem:plargeAlpha}
 If $\alpha > 1/2$ then $p > 1-\alpha$.
\end{lemma}
\begin{proof}
 We prove the lemma by showing that for $\alpha > 1/2$ and $p = 1-\alpha$, the value of the argument of $W_i$ in \eqref{eq:GwithSomeW} is less than $-e^{-1}$ for $s \in (1-\varepsilon_\alpha,1)$ for some $\varepsilon_\alpha > 0$. Since $W_{-1}(s)$ and $W_0(s)$, the real branches of the W-function, are only defined for $s \geq -e^{-1}$, together with Lemma \ref{lem:mustBeTheW} this gives us a contradiction.
 
 Indeed, let
 \begin{align*}
  g_{p}(s) & = \alpha s - \alpha -1 + \left ( 1-s^{-1} \right ) p, \\
  h_{p}(s) & = -s^{-1} \exp \left (g_{p}(s) \right ),
 \end{align*}
 so that \eqref{eq:GwithSomeW} can be rewritten as $G(s) = -s W_i (h_{p}(s))$ for some $i = i(s) \in \{0,-1\}$.
 
 We clearly have $g_{p}(1) = -1$ and $h_{p}(1) = -e^{-1}$. Also,
 \begin{equation}
 \label{eq:hDerivative}
  h'_{p}(s) = \exp \left (g_{p}(s) \right ) \left ( s^{-2} - s^{-1} \left (\alpha + p s^{-2}  \right ) \right ),
 \end{equation}
 which implies that $h'_{1-\alpha}(1) = 0$. We also see that
 \begin{align*}
  h''_{p}(s) & = \exp \left (g_{p}(s) \right ) \left ( -2s^{-3} +  \alpha s^{-2} + 3 p s^{-4}  + (\alpha + p s^{-2}  ) \left ( s^{-2} - s^{-1} \left (\alpha + p s^{-2}  \right )  \right ) \right ) \\
             & = \exp \left (g_{p}(s) \right ) \left ( - \alpha^2 s^{-1}  + 2\alpha s^{-2} - (2 + 2\alpha p ) s^{-3} + 4p s^{-4}  - p^2 s^{-5}  \right ).
 \end{align*}
 This gives
 \begin{align*}
  h''_{1-\alpha}(1) & = e^{-1} (-\alpha^2 + 2\alpha - 2 - 2\alpha + 2\alpha^2 + 4 - 4\alpha -1 + 2\alpha - \alpha^2) \\
                    & = e^{-1} (1 - 2\alpha ) < 0
 \end{align*}
for $\alpha > 1/2$. Hence, as clearly $h'''_{1-\alpha}(s) < \infty$ around $s=1$, $h_{1-\alpha}(s) < -e^{-1}$ for $s<1$ large enough. This completes the proof of the lemma.
\end{proof}

Since for $\alpha > 1/2$ we have $p > 1-\alpha$, let us again look at $s^* = (1-p)/\alpha \in (0,1)$. We have $g_p(s^*) = -(s^*)^{-1}$ and so $h_p(s^*) = -(s^*)^{-1} \exp(-(s^*)^{-1})$.  By Fact \ref{fact:ex^x}, we see that 
\[
f_{-1}(s^*) = -s^* W_{-1}(-(s^*)^{-1} \exp(-(s^*)^{-1})) = 1
\]
and since a probability generating function may not take the value 1 for $s \in (0,1)$, we cannot have $G(s^*) = f_{-1}(s^*)$. Hence we must have $G(s^*) = f_{0}(s^*)$. In the following lemma we prove a considerably stronger result about the structure of $G(s)$ when $\alpha > 1/2$.

\begin{lemma}
\label{lem:branchSwitch}
Let $\alpha > 1/2$. Then there is some $s' \in (0, s^*)$ such that $G(s) = f_{-1}(s)$ if $s<s'$ and $G(s) = f_0(s)$ if $s \geq s'$.
\end{lemma}
\begin{proof}
We prove the lemma by analysing the function $h_{p}(s)$ defined in the proof of Lemma \ref{lem:plargeAlpha}. Since for $\alpha > 1/2$ we cannot have $G(s) = f_{-1}(s)$ for all $s \in (0,1)$, there must be some $s' \in (0,1)$ such that $h_{p}(s') = -e^{-1}$ (as this is the only way for the two branches of the Lambert W-function to meet in $(0,1))$. In fact, $s'$ must be a turning point for $h_{p}(s)$ to make sure that we have a real solution for all $s \in (0,1)$.

By \eqref{eq:hDerivative}, we immediately see that there are at most two real solutions to $h'_{p}(s) = 0$. Hence $h_{p}(s)$ has at most two turning points in $(0,1)$, and since we also have $h_p(1)=-e^{-1}$, $s'$ is the only solution to $h_{p}(s') = -e^{-1}$ in $(0,1)$. By Lemma \ref{lem:W-1AtS=0} we have that $G(s) = f_{-1}(s)$ for $s \in (0,\varepsilon_\alpha)$, and we know that $G(s^*) = f_0(s^*)$, so this implies that $G(s) = f_{-1}(s)$ for $s<s'$ and $G(s) = f_0(s)$ for $s \geq s'$.
\end{proof}

\begin{cor}
\label{cor:p0largeAlphaBounds}
Let $\alpha > 1/2$. Then $p \in (1-\alpha, \frac{1}{4 \alpha})$.
\end{cor}
\begin{proof}
 We have $p > 1-\alpha$ by Lemma \ref{lem:plargeAlpha}. We also know that for $\alpha > 1/2$ the two functions $f_{-1}(s)$ and $f_{0}(s)$ must meet in $(0,1)$, and so there is some $s' \in (0,1)$ such that $h_{p}(s') = -e^{-1}$ and $s'$ is a turning point for $h_{p}(s)$. However, we also must have $h_{p}(1) = -e^{-1}$, as $G(1) = -W_i (h_{p}(1)) = 1$. Hence $h_{p}(s)$ must have two turning points in $(0,1)$, which by \eqref{eq:hDerivative} implies that there must be two solutions to
 \[
  \alpha s^2 - s + p = 0.
 \]
 This implies that $1 - 4 \alpha p > 0$, and the bound $p < \frac{1}{4 \alpha}$ follows.
\end{proof}

\begin{cor}
\label{cor:s'Value}
 The value of $s'$ in Lemma \ref{lem:branchSwitch} is
 \[
  s' = \frac{1-\sqrt{1-4 p \alpha}}{2 \alpha}.  
 \]
\end{cor}
\begin{proof}
 Proceeding as in the proof of Corollary \ref{cor:p0largeAlphaBounds}, we see that the turning points of $h_{p}(s)$ are $s_1 = \frac{1-\sqrt{1-4 p \alpha}}{2 \alpha}$ and $s_2 = \frac{1+\sqrt{1-4 p \alpha}}{2 \alpha}$ (notice that for $p > 1-\alpha$ we have $s_1, s_2 \in (0,1)$). Now, as we discussed above, we must have $h_{p}(s_1) = -e^{-1}$ and $h_{p}(s_2) > -e^{-1}$. Consequently, we have $f_{-1}(s_1) = f_{0}(s_1)$.
\end{proof}

In the following corollary let us finally summarise what we can say about the value of $p$ in the case $\alpha > 1/2$.

\begin{cor}
\label{cor:pequations}
 For $\alpha > 1/2$, taking $s' = \frac{1-\sqrt{1-4 p \alpha}}{2 \alpha}$, the value of $p \in (1-\alpha, \frac{1}{4 \alpha})$ satisfies $h_{p} (s') = -e^{-1}$.
\end{cor}

 \begin{figure}[htb]
    {\includegraphics[width=.7\textwidth]{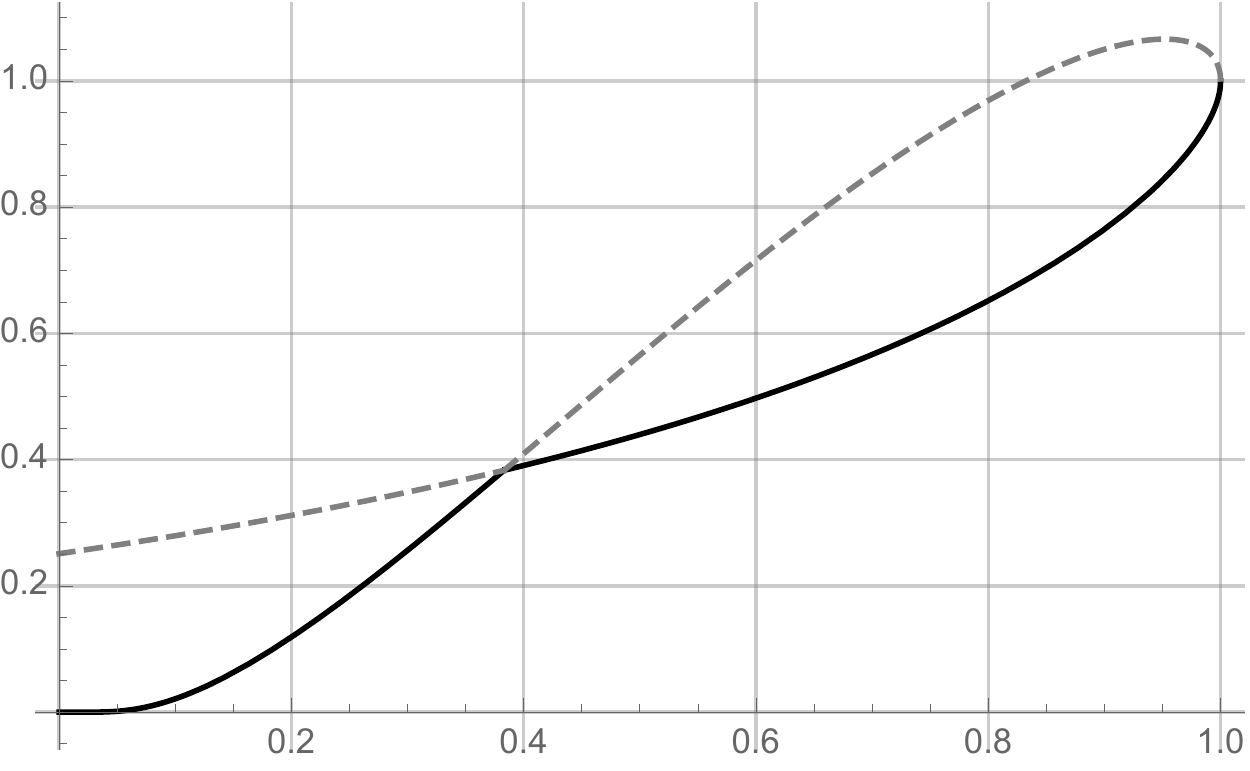}}
    \caption{The graphs of $f_0(s)$ (black solid curve) and $f_{-1}(s)$ (grey dashed curve) for $\alpha = 0.9$ and $p = 0.251042$, giving $s' \approx 0.3832$.}
    \label{fig:largeAlpha}
 \end{figure}

Equipped with Lemma \ref{lem:branchSwitch} and the above corollaries, we can understand the behaviour of $G(s)$ when $\alpha > 1/2$. Since we do not have an analytic expression for $p$ in that case, Figure \ref{fig:largeAlpha} shows an approximation of the probability generating function of $X$ when $\alpha = 0.9$, in which case we obtain $p \approx 0.251042$ and $s' \approx 0.3832$.

\section{Generalisations}
\label{sec:generalisations}

Consider our parking process on a $\PGW(1)$ tree. There are two aspects of this model which one might think of generalising: the distribution of the number of cars arriving at each vertex, and the offspring distribution of the Galton-Watson process, i.e.\ the laws of $P$ and $N$ respectively.  
One specific such situation, which we shall summarise below, has been studied by Jones~\cite{Jones} in the context of a model for rainfall runoff down a hill.  (We emphasise that the results in our papers were obtained independently, and it was only by a happy accident that we became aware of Jones' work.)  
We will then give a brief overview of the sorts of generalisations that one might expect in the situations of subcritical, critical and supercritical offspring distributions respectively.  We do not attempt an exhaustive survey here, but rather defer that to future work.  We focus on the random variable $X$ and potential analogues of the phase transition (\ref{eqn:discphtr}).  We think of the parking process as a dependent version of site percolation, where vertices for which $X > 0$ are occupied.

Before we discuss generalisations, we remind the reader of an important result due to Kesten, to which we will shortly make appeal.

\begin{thm}[Kesten~\cite{Kesten}] \label{thm:Kesten}
Suppose that $(Z_n)_{n \ge 0}$ is a Galton-Watson process with offspring distribution $\nu$ such that $\nu(0) < 1$ and $\mu = \sum_{k=1}^{\infty} k \nu(k) \le 1$.  Let $T$ be the associated family tree.  Then if $T_n$ is distributed as $T$ conditioned on the event $\{Z_n > 0\}$, we have
\[
T_n \convdist T_{\infty},
\]
as $n \to \infty$, in the sense of local weak convergence, where $T_{\infty}$ is the random tree constructed as follows.  First, take an infinite path labelled by $\{1,2,3,\ldots\}$, rooted at $1$. To each node along the path, attach an independent random number of children, with distribution $\hat{\nu}(k) = (k+1) \nu(k+1)/\mu$, $k \ge 0$.  Then attach an independent Galton-Watson tree with offspring distribution $\nu$ rooted at each of these neighbours of the infinite path.
\end{thm}

In the case where $\nu$ is a Poisson distribution we have $\hat{\nu} = \nu$ and so this spine decomposition has the particularly simple form we exploited earlier in the paper.

\subsection{Binary branching, paired arrivals} \label{subsec:Owen}
We turn now to Jones' results from \cite{Jones}.  He takes the offspring distribution to be
\[
\Prob{N = 0} = \beta, \quad \Prob{N = 1} = 1-2\beta, \quad \Prob{N=2} = \beta,
\]
where $\beta \in (0,1/4]$, and the arrival distribution to be
\[
\Prob{P = 0} = 1- \alpha/2, \quad \Prob{P = 2} = \alpha/2,
\]
where $\alpha \in (0,2)$, so that we have $\E{P} = \alpha$. (Our parameterisation differs from the one used in \cite{Jones} to provide an easier comparison with the results of Section~\ref{sec:intro}.) Note that the offspring distribution is critical for all values of $\beta$.  Jones observes completely analogous phenomena to those we have discussed above.  Specifically, for each $\beta \in (0,1/4]$, let
\begin{equation}
\label{eq:JonesCrit}
\alpha_c(\beta) = 1 + \beta - \sqrt{\beta(2 + \beta)}.
\end{equation}
Then
\begin{equation}
\label{eq:JonesExp}
 \E{X} = \begin{cases}
  \frac{1 - \alpha + 2 \alpha \beta - \sqrt{1-2\alpha(1-\alpha/2 + \beta)}}{2\beta} & \text{for $\alpha \le \alpha_c(\beta)$}\\
  \infty & \text{for $\alpha > \alpha_c(\beta)$}.
 \end{cases}
\end{equation}
(Jones formulates his results in terms of the random variable $W = (X-1)^+$ but it is relatively straightforward to translate between the two situations.) For $\beta = 1/4$, for example, we get $\alpha_c(1/4) = 1/2$ and at the point of the phase transition the mean is $\E{X} = 3/2$.

Strikingly, Jones observes the same ``branch-switching'' phenomenon in the supercritical phase as we do.  The probability generating function $G(s) = \E{s^X}$ satisfies a quadratic equation to which there are two possible solutions: in the subcritical phase, one of them gives the generating function for all $s \in [0,1]$; in the supercritical phase, the generating function follows one branch at the start of the interval and the other from a point in the middle of the interval.

Jones also considers what happens in the tree conditioned to be infinite.  By Theorem~\ref{thm:Kesten}, we have an infinite spine to each point of which we attach an extra edge (leading to an independent copy of the unconditioned tree) with probability $\hat{\nu}(1) = 2\beta$ and no edge otherwise. An analogous random walk argument leads to a finite expected number of cars at the root if and only if
\[
\E{P} -1 + \E{\hat{N}} \E{(X-1)^+} < 0,
\]
where $\hat{N}$ is a random variable with law $\hat{\nu}$ having expectation
\[
 \mathbb{E}[\hat{N}] = \sum_{k \geq 0} \frac{k (k+1) \Prob{N=k+1}}{\E{N}} = \frac{\E{N^2}-\E{N}}{\E{N}} = \E{N^2}-1.
\]
In other words, the expected number of cars at the root is finite iff
\[
\E{X} < \frac{\E{P} \var{N} + 1 - \E{P}}{\var{N}} = \frac{1 - \alpha + 2 \alpha \beta}{2 \beta},
\]
which by \eqref{eq:JonesExp} and \eqref{eq:JonesCrit} occurs iff $\alpha < \alpha_c(\beta)$.  We emphasise that, as in the Poisson case, the critical point is the same for the conditioned and unconditioned trees.

(Jones also partly generalises his results to arbitrary arrival distributions with the same binary branching but we will not give the details here.)

\subsection{Subcritical branching}

For completeness, we now show that a phase transition of the form (\ref{eqn:discphtr}) for $\E{X}$ cannot occur if the offspring distribution is subcritical.

\begin{prop}
Let $\lambda = \E{N}$.  If $\lambda < 1$ then $\E{X} < \infty$ for all $\alpha \ge 0$.  
\end{prop}

\begin{proof}
Write $Q$ for the total progeny of the branching process.  Then it is elementary that $\E{Q} = \frac{1}{1-\lambda}$.  Now observe that we have the crude bound $X \le \sum_{i=1}^{Q} P_i$ and that the right-hand side has expectation $\frac{\alpha}{1-\lambda}$ which is finite for all $\alpha \ge 0$.
\end{proof}

\subsection{Critical branching}
Now suppose that we fix an offspring distribution such that $\lambda = \E{N} =1$ and $\var{N} < \infty$, and assume that $\var{P} < \infty$.   

Let us make the (unjustified) hypothesis that $\var{X} < \infty$ whenever $\E{X} < \infty$.  Then, using the RDE (\ref{eqn:rde}) and considering the variances of the two sides, we see that
\[
\var{X} = \var{P} + \var{(X-1)^+} + \E{(X-1)^+}^2 \var{N}.
\]
After rearrangement and cancellation this yields a quadratic equation for $\E{X}$:
\[
0 = \var{N} \E{X}^2 - 2(1 -\alpha + \alpha \var{N}) \E{X} + \var{P} + \alpha + \alpha^2(\var{N} - 1).
\]
The discriminant is
\[
\Delta = 4\left( 1 - 2 \alpha + \alpha^2 + \var{N}(\alpha - \alpha^2 - \var{P}) \right),
\]
and this quantity must be non-negative in order to obtain a meaningful value for $\E{X}$.  Assuming this to be the case then there are \emph{a priori} two possible values for $\E{X}$:
\[
\frac{1 - \alpha + \alpha \var{N} \pm \sqrt{1 - 2 \alpha + \alpha^2 + \var{N}(\alpha - \alpha^2 - \var{P}) }}{\var{N}}.
\]
In both the Poisson case we study in this paper, and the situation studied by Jones, we take the smaller root, and this value is correct all the way up to the phase transition.

In order to meaningfully talk about a phase transition in a more general setting, we need a family of distributions for $P$, parameterised by $\alpha = \E{P}$ for $\alpha \ge 0$.  Again we assume $\var{P} < \infty$ and write $h(\alpha) = \var{P} + \alpha^2 - \alpha = \E{P^2} - \alpha$.  Note that as $P$ takes non-negative integer values, $P(P-1) \ge 0$, and so $h(\alpha) \ge 0$.  Observe also that $h(0) = 0$.  We will make the natural assumption that $P$ is stochastically increasing in $\alpha$ which entails that $h(\alpha) = \E{P(P-1)}$ is an increasing function.  

We must then have that $\E{X}$ is increasing as a function of $\alpha$.  The function $\alpha \mapsto (1 - \alpha)^2 - \var{N}h(\alpha)$ is decreasing on $[0,1]$. So if $\var{N} \le 1$, the numerator can only be an increasing function if we take the smaller root.  This argument leads us to make the following conjecture.

\begin{conj} \label{conj:phtr}
Suppose that $\lambda = 1$ and that $\var{N} \le 1$.  Suppose that $P$ is stochastically increasing in $\alpha$ and that $\var{P} < \infty$ for all $\alpha \ge 0$.  Define 
\[
\alpha_c = \inf\left\{\alpha \ge 0: \alpha = 1 - \sqrt{\var{N} h(\alpha)}\right\}.
\]
Then 
\[
\E{X} =  \begin{cases}
\frac{1 - \alpha + \alpha \var{N} - \sqrt{(1 - \alpha)^2 - \var{N}h(\alpha)}}{\var{N}} & \text{ if $\alpha \le \alpha_c$} \\
\infty & \text{ if $\alpha > \alpha_c$.}
\end{cases}
\]
\end{conj}

We conjecture that the jump from $\E{X} < \infty$ to $\E{X} = \infty$ coincides with the onset of long-range dependence in the model: above $\alpha_c$, the occupied cluster of the root appears to become macroscopic in the sense that it occupies a positive fraction of the tree.  Since the size of the tree has infinite expectation, this gives that $X$ also has infinite expectation.

Consider now the tree conditioned to be infinite, work under the conditions of Conjecture~\ref{conj:phtr} and suppose that the conjecture is true.  Then the same argument as in Section~\ref{subsec:Owen} gives that, if $\tilde{X}$ is the number of cars visiting the root of the conditioned tree, we have $\mathbb{E}[\tilde{X}] < \infty$ iff 
\[
\E{X} < \frac{1 - \alpha + \alpha \var{N}}{\var{N}},
\]
which occurs iff $\alpha < \alpha_c$.

\subsection{Supercritical branching}
Finally, let us consider the situation where $\lambda = \E{N} > 1$. Let $\E{P} = \alpha$ as usual. The first difference we immediately observe here is that an analogue of Proposition \ref{prop:p0} gives us
\[
 \E{X} = \frac{\lambda-\alpha - \lambda p_\lambda}{\lambda-1},
\]
where $p_\lambda = \Prob{X=0}$, whenever $\E{X}$ is finite. Observe that the assumption that $\E{X}$ is finite does not give us an explicit formula for $p_\lambda$. On the other hand, we can always bound $\E{X}$ from above by $\frac{\lambda-\alpha}{\lambda-1}$. Thus we see that as $\alpha$ increases from $0$, $\E{X}$ undergoes a discontinuous phase transition from a bounded value to $\infty$. In fact a stronger statement, found in the following theorem, is true. 

\begin{thm} \label{thm:supercritical}
Suppose that $\E{N} = \lambda > 1$ and that $P$ is stochastically increasing in $\alpha = \E{P}$. Then there exists $\alpha_c \in (0,1)$ such that if $\E{P} = \alpha < \alpha_c$ then $\E{X} < \frac{\lambda-\alpha}{\lambda-1}$, while if $\alpha > \alpha_c$ then, conditionally on the non-extinction of the tree, $X = \infty$ almost surely.
\end{thm}
\begin{proof}
 As already discussed, if $\alpha$ is such that $\E{X} > \frac{\lambda-\alpha}{\lambda-1}$ then $\E{X} = \infty$. Let $\alpha_c$ be the supremum of the set of $\alpha$ for which $\E{X}$ is finite. We need to show that for $\alpha > \alpha_c$ we have $\Prob{X = \infty \ | \ |T|=\infty} = 1$.

 Observe that $\Prob{X = \infty \ | \ |T|=\infty}$ is equal to either $0$ or $1$, as when this event has positive probability, there almost surely exists some vertex of the tree which is visited by infinitely many cars, and then the same must be true of the root. (On the other hand, if $\{|T| < \infty\}$ has positive probability then, conditionally on this event, $|T|$ has finite mean.  So then $\E{X | |T| < \infty} < \infty$ by the same argument as in the subcritical case.)
 
Let $T$ be the tree with offspring distribution $N$. Assume first that $\Prob{N=0}=0$ so that $|T| = \infty$ almost surely. Since $\lambda = \E{N} > 1$, we also have $\Prob{N>1} = \beta >0$. Choose an arbitrary path $(v_0, v_1, v_2, \ldots)$ from the root $v_0$ of the tree to infinity, without revealing the rest of the tree. Observe that every $v_i$ has at least one additional child (other than $v_{i+1}$) with probability $\beta$.
 
For $i \geq 0$, let $X_i$ be defined as follows. If $v_i$ has no other child but $v_{i+1}$, set $X_i = 0$. Otherwise, let $w_i$ be an arbitrary child of $v_i$ other than $v_{i+1}$. Next, let $Y_i$ be the number of cars that arrive at $w_i$ in the usual parking process on the subtree of $T$ rooted at $w_i$, and let $X_i = (Y_i-1)^+$. By assumption, we have $\E{Y_i} = \infty$, so also $\E{X_i | w_i \mbox{ exists}} = \infty$. Hence,
 \[
  \E{X_i} = \beta \E{X_i | w_i \mbox{ exists}} = \infty.
 \]
 Thus by the random walk interpretation of the parking process on a path, and by coupling the original parking process on $T$ with the process we describe above, we see that the number $X$ of cars that arrive at the root is infinite almost surely.
 
 Now, assume that $\Prob{N=0} > 0$ and let $q = \Prob{|T| < \infty}$. As $\Prob{N=0} > 0$ and $\E{N} > 1$, we have $0 < q < 1$. Conditioned on $\{|T| = \infty\}$, the distribution of $T$ is that of a multitype Galton-Watson tree $\tilde{T}$ with vertices of two types, $s$ and $e$. The root of $\tilde{T}$ is of type $s$. A vertex of type $s$ produces $S$ children of type $s$ and $E$ children of type $e$, with probability generating function $G(x, y) = \E{x^S y^E}$ given by
\[
G(x,y) = \frac{G_N((1-q)x+qy)-G_N(qy)}{1-q}.
\]
Most importantly, the probability that a vertex of type $s$ has no children of type $s$ is given by
\[
G(0,1) = \frac{G_N(q)-G_N(q)}{1-q} = 0.
\]
Moreover,
\[
\frac{\partial}{\partial x} G(x,1) = \frac{G_N'((1-q)x+q)(1-q)}{1-q} = G_N'((1-q)x+q),
\]
which for $x=1$ is equal to $G_N'(1) = \E{N} > 1$. On the other hand, the vertices of type $e$ produce only children of type $e$, and the subtrees rooted at vertices of type $e$ are subcritical with offspring distribution $N_e$ given by $\Prob{N_e = k} = q^{k-1}\Prob{N=k}$ for $k \geq 0$. (For more on the distributions of conditioned Galton-Watson trees see Abraham and Delmas \cite{AbrahamDelmas-GWintroduction}.)

To complete the proof, we now look at the parking process on the subtree of $\tilde{T}$ induced by the vertices of type $s$. By the above, these vertices form a supercritical Galton-Watson tree with offspring distribution $N_s$ satisfying $\Prob{N_s = 0} = 0$. Hence, we are back in the case we have already analysed and, by coupling the parking process limited to this subtree with the original process, we see that we again have $X = \infty$ almost surely.
\end{proof}

In the following proposition we discuss a natural example of the parking process in the supercritical setting: the complete infinite binary tree, with the distribution of the car arrivals concentrated on the values 0 and 2 only.  In this case, we are able to provide bounds on the critical value $\alpha_c$.

\begin{prop} \label{prop:binary}
 For the complete binary tree (i.e.\ $\Prob{N=2} = 1$) with arrival distribution
 \[
  \Prob{P=2} = \alpha/2, \quad \Prob{P = 0} = 1-\alpha/2,
 \]
there exists $\alpha_c \in [1/32,1/2]$ such that if $\alpha < \alpha_c$ then $\E{X} < 2-\alpha$, while if $\alpha > \alpha_c$ then $X = \infty$ almost surely.
\end{prop}
\begin{proof}
 By Theorem \ref{thm:supercritical} we know that we either have $\E{X} < 2-\alpha$ or $X = \infty$ almost surely. Let us show that for $\alpha > 1/2$ the latter holds. Consider first only the vertices in the ``even'' generations of the tree (with the root being the 0th generation), with edges ``inherited'' from the original tree (so that every vertex is adjacent to its four grandchildren). This gives a complete quaternary tree.  Consider now the set of vertices in this quaternary tree at which there are non-zero arrivals. For $\alpha/2 > 1/4$, there is an infinite path of initially occupied vertices. Observe that these vertices on their own give us an infinite eventually occupied path in the original tree, as the vertices in even generations on the path each have $P=2$. However, infinitely many of the vertices in odd generations on this path will also be initially occupied almost surely which implies that infinitely many cars will arrive at the starting vertex of the path, and so also at the root of the tree. Thus $X=\infty$ almost surely in this case.
 
 Now assume that $\alpha < 1/32$. We want to show that the eventually occupied cluster of the root is finite with positive probability. This implies that $X<\infty$ with positive probability, which in turn gives us $X < \infty$ almost surely, and so also $\E{X} < 2-\alpha$. If the cluster of eventually occupied vertices containing the root is infinite then for any $M$, there is some $n \geq M$ and a set $A$ of initially occupied vertices of size at least $n/2$ (as $P=2$ for an initially occupied vertex) such that the cars arriving in $A$ on their own occupy a cluster of size $n$ containing the root in the final configuration.
 
 Such a cluster of size $n$, together with all the immediate descendants of its vertices, forms a binary tree with $n+1$ leaves. It is well known that the number of such trees is equal to the $n$th Catalan number
 \[
  C_n = \frac{1}{n+1}\binom{2n}{n} < 4^n.
 \]
 There are $\binom{n}{\lfloor n/2 \rfloor} < 2^n$ ways to choose the set $A$. Therefore, the probability of the event that such a cluster of size $n$ can be found is at most
 \[
  \sum_{n=M}^{\infty} C_n \binom{n}{\lceil n/2 \rceil} (\alpha/2)^{n/2} <\sum_{n=M}^{\infty} 4^n 2^n (\alpha/2)^{n/2} < \sum_{n=M}^{\infty} \left ( (32 \alpha)^{1/2} \right )^n = \frac{\left( (32 \alpha)^{1/2} \right )^M}{1-(32 \alpha)^{1/2}} < 1
 \]
 for $\alpha < 1/32$ and $M = M_\alpha$ large enough. This completes the proof of the proposition.
\end{proof}

\section{Acknowledgments}

We are very grateful to Marie-Louise Lackner for introducing us to the problem, and to Owen Jones for telling us about his work and allowing us to see his manuscript \cite{Jones}. C.G.'s research is supported by EPSRC Fellowship EP/N004833/1.

\bibliographystyle{abbrv}

\begin{thebibliography}{10}

\bibitem{AbrahamDelmas-GWintroduction}
R.~Abraham and J.-F. Delmas.
\newblock An introduction to {G}alton-{W}atson trees and their local limits.
\newblock Lecture notes available at \texttt{arXiv:1506.05571}.

\bibitem{Addario-Berry}
L.~Addario-Berry.
\newblock The local weak limit of the minimum spanning tree of the complete
  graph.
\newblock Preprint available at \texttt{arXiv:1301.1667}.

\bibitem{AldousSteele-LocalConvergence}
D.~Aldous and J.~Steele.
\newblock The objective method: Probabilistic combinatorial optimization and
  local weak convergence.
\newblock In H.~Kesten, editor, {\em Probability on Discrete Structures},
  volume 110 of {\em Encyclopaedia of Mathematical Sciences}, pages 1--72.
  Springer, 2004.

\bibitem{AldousBandyopadhyay}
D.~J. Aldous and A.~Bandyopadhyay.
\newblock A survey of max-type recursive distributional equations.
\newblock {\em Ann. Appl. Probab.}, 15(2):1047--1110, 2005.

\bibitem{BarlowKumagai}
M.~T. Barlow and T.~Kumagai.
\newblock Random walk on the incipient infinite cluster on trees.
\newblock {\em Illinois J. Math.}, 50(1-4):33--65 (electronic), 2006.

\bibitem{BenjaminiSchramm}
I.~Benjamini and O.~Schramm.
\newblock Recurrence of distributional limits of finite planar graphs.
\newblock {\em Electron. J. Probab.}, 6:13 pp., 2001.

\bibitem{BrownPekozRoss-skipFreeWalks}
M.~Brown, E.~Pek\"oz, and S.~Ross.
\newblock Some results for skip-free random walk.
\newblock {\em Probab. Eng. Inform. Sc.}, 24:491--507, 2010.

\bibitem{CGHJK-Lambert}
R.~M. Corless, G.~H. Gonnet, D.~E.~G. Hare, D.~J. Jeffrey, and D.~E. Knuth.
\newblock On the {L}ambert {$W$} function.
\newblock {\em Adv. Comput. Math.}, 5(4):329--359, 1996.

\bibitem{Grimmett-TreeLimit}
G.~Grimmett.
\newblock Random labelled trees and their branching networks.
\newblock {\em J. Austral. Math. Soc.}, 30:229--237, 1980.

\bibitem{Jones}
O.~Jones.
\newblock Runoff on rooted trees.
\newblock Preprint (personal communication), 2016.

\bibitem{Kesten}
H.~Kesten.
\newblock Subdiffusive behavior of random walk on a random cluster.
\newblock {\em Ann. Inst. H. Poincar{\'e} Probab. Statist.}, 22(4):425--487,
  1986.

\bibitem{KonheimWeiss-parkingPath}
A.~Konheim and B.~Weiss.
\newblock An occupancy discipline and applications.
\newblock {\em SIAM J. Appl. Math.}, 14:1266--1274, 1966.

\bibitem{LacknerPanholzer}
M.-L. Lackner and A.~Panholzer.
\newblock Parking functions for mappings.
\newblock {\em J. Combin. Theory Ser. A}, 142:1--28, 2016.

\bibitem{LuczakWinkler}
M.~Luczak and P.~Winkler.
\newblock Building uniformly random subtrees.
\newblock {\em Random Structures Algorithms}, 24(4):420--443, 2004.

\bibitem{LyonsPeledSchramm}
R.~Lyons, R.~Peled, and O.~Schramm.
\newblock Growth of the number of spanning trees of the
  {E}rd{\H{o}}s-{R}{\'e}nyi giant component.
\newblock {\em Combinatorics, Probability and Computing}, 17:711--726, 9 2008.

\bibitem{Stanley-hyperplanesAndTrees}
R.~Stanley.
\newblock Hyperplane arrangements, interval orders, and trees.
\newblock {\em Proc. Natl. Acad. Sci.}, 93:2620--2625, 1996.

\bibitem{Stanley-parkingPartitions}
R.~Stanley.
\newblock Parking functions and noncrossing partitions.
\newblock {\em Electron. J. Combin.}, 4:1--14, 1997.

\bibitem{Stanley-enumerativeCombinatorics}
R.~Stanley.
\newblock {\em Enumerative Combinatorics}, volume I \& II.
\newblock Cambridge University Press, 1997 \& 1999.

\bibitem{Stanley-hyperplanesAndParking}
R.~Stanley.
\newblock Hyperplane arrangements, parking functions and tree inversions.
\newblock In B.~Sagan and R.~Stanley, editors, {\em Mathematical Essays in
  honor of {G}ian-{C}arlo {R}ota}, volume 161 of {\em Progress in Mathematics},
  pages 359--375. Springer, 1998.

\end{thebibliography}

\end{document}